\documentclass{article}
\usepackage{amsfonts}

\usepackage{amsmath}
\usepackage{chicago}

\newtheorem{theorem}{Theorem}

\newtheorem{lemma}[theorem]{Lemma}

\newenvironment{proof}[1][Proof]{\noindent\textbf{#1.} }{\ \rule{0.5em}{0.5em}}
\input{tcilatex}

\begin{document}

\title{A converse comparison theorem for backward stochastic differential
equations with jumps\thanks{%
This paper is based on a homonymous paper published in \textit{Statistics
and Probability Letters} 81, 298--301, 2011, which contains an error in the
proof of the main theorem.\ It presents a similar result, but for a more
restricted class of equations, so that the error is now fixed. The author
thank Monique Jeanblanc for pointing out the error and for fruitful
discussions about how to fix it. }}
\author{Xavier De Scheemaekere\thanks{%
F.R.S.-F.N.R.S. Research Fellow; Centre Emile Bernheim, Solvay Brussels
School of Economics and Management, Universit\'{e} libre de Bruxelles\textit{%
\ }(ULB); Postal address: Av. F.D. Roosevelt, 50, CP 145/1, 1050 Brussels,
Belgium ; Tel.: +32.2.650.39.59; E-mail address: xdeschee@ulb.ac.be.}}
\date{May 2011 }
\maketitle

\begin{abstract}
This paper establishes a converse comparison theorem for real-valued
decoupled forward backward stochastic differential equations with jumps.
\end{abstract}



\section{Introduction}

A backward stochastic differential equation (BSDE) is an equation of the type%
\begin{equation}
Y_{t}=\xi
_{T}+\dint\nolimits_{t}^{T}g(s,Y_{s},Z_{s})ds-\dint%
\nolimits_{t}^{T}Z_{s}dW_{s},  \label{classical bsde}
\end{equation}%
where $\xi $ is called the terminal condition, $T$ the terminal time, $g$
the generator, and the pair of processes $(Y,Z)$ the solution of the
equation. In 1990, Pardoux and Peng \cite{Pardoux and Peng} proved that
there exists a unique adapted and square integrable solution to BSDE (\ref%
{classical bsde}) as soon as the generator is Lipschitz with respect to $y$
and $z$ and the terminal condition is square integrable.\ Since then, BSDE
theory has become an important field of research, with applications, e.g.,
in stochastic optimal control, mathematical finance or partial differential
equations (PDEs).

An important feature of BSDE theory is the comparison theorem (Peng \cite{cc
Peng}, El Karoui et al. \cite{ccEl Karoui et al}), which plays the same role
as the maximum principle for PDEs.\ The comparison theorem allows to compare
the solutions of two real-valued BSDEs whenever we can compare the terminal
conditions and the generators. The converse, however, is not true in
general. Hence, converse comparison theorems for BSDEs have been concerned
with the following question: If one can compare the solutions of two
real-valued BSDEs with the same terminal condition, for all terminal
conditions\textit{, }can one compare the generators?

Previous works on the subject (Chen \cite{Chen}, Briand et al. \cite{Coquet
et al}, Coquet et al. \cite{Coquet Hu et al}, Jiang \cite{Jiang I},\cite{cc
Jiang}) deal with classical BSDEs (i.e., without jumps).\ This paper
provides a converse comparison theorem for class of BSDEs with jumps, namely
infinite horizon decoupled forward backward stochastic differential
equations with jumps.\ 

\section{BSDEs with jumps}

\subsection{Notation and assumptions}

Suppose that $W_{t}^{\top }=(W_{t}^{1},...,W_{t}^{d}),$ $t\geq 0,$ is a $d$%
-dimensional standard Brownian Motion and $K^{\top
}(t)=(K_{1}(t),...,K_{l}(t))$, $t\geq 0,$ is an $l$-dimensional stationary
Poisson point process taking values in a measurable space $(B,\mathbf{B})$,
where $B=%
\mathbb{R}
_{0}$ is equipped with its Borel field $\mathbf{B.}$ Denote by $\mu ^{\top
}(dt,de)=(\mu _{1}(dt,de),...,\mu _{l}(dt,de))$ the Poisson random measure
induced by $K(\cdot )$ with compensator $\lambda ^{\top }(de)dt=(\lambda
_{1}(de),...,\lambda _{l}(de))dt$ such that $\left\{ \widetilde{\mu }%
_{i}([0,t],A)=(\mu _{i}-\lambda _{i})([0,t],A)\right\} _{t\geq 0}$ is a
martingale for all $A\in \mathbf{B}$ satisfying $\lambda _{i}(A)<\infty ,$ $%
i=1,...,l$. $\lambda _{i}$ is assumed to be a $\sigma $-finite measure on $%
(B,\mathbf{B})$.

Let $(\Omega ,\mathbf{\tciFourier },\left( \tciFourier _{t}\right) _{t\geq
0},P)$ be a complete and standard measurable probability space equipped with
a filtration denoted as 
\begin{equation*}
\left( \tciFourier _{t}\right) _{t\geq 0}=\left( \sigma \lbrack W_{s};s\leq
t]\vee \sigma \lbrack \mu ((0,s],A);s\leq t,A\in \mathbf{B}]\vee N\right)
_{t\geq 0},
\end{equation*}%
where $N$ is all the $P$-null sets.\ It is a complete right-continuous
filtration. Let $\tau $ be a (possibly infinite) $\left( \tciFourier
_{t}\right) _{t\geq 0}$-stopping time and\ define the following spaces:

\begin{itemize}
\item $S^{2}$ denotes the set of all $\tciFourier _{t}$-progressively
measurable RCLL processes $Y$ valued in $%
\mathbb{R}
^{n}$ such that $E\left[ \underset{0\leq t\leq \tau }{\sup }\left\vert
Y_{t}\right\vert ^{2}\right] <\infty ;$

\item $L^{2}(W)$ denotes the set of all predictable\footnote{%
By predictable, we mean predictable with respect to the filtration $\left(
\tciFourier _{t}\right) _{0\leq t\leq \tau }.$} processes $Z$ valued in $%
\mathbb{R}
^{n\times d}$ such that $E\left[ \tint\nolimits_{0}^{\tau }\left\vert
Z_{t}\right\vert ^{2}dt\right] <\infty ;$

\item $L^{2}(\widetilde{\mu })$ denotes the set of all $\mathbf{P}\otimes 
\mathbf{B}$-measurable processes $U_{t}(e)$ valued in $%
\mathbb{R}
^{n\times l}$ such that%
\begin{equation*}
E\left[ \dint\nolimits_{0}^{\tau }\dint\nolimits_{B}\left\vert
U_{t}(e)\right\vert ^{2}\lambda (de)dt\right] :=E\left[ \dint\nolimits_{0}^{%
\tau }\left\Vert U_{t}(\cdot )\right\Vert ^{2}dt\right] <\infty ,
\end{equation*}
\end{itemize}

where \textbf{P} denotes the $\sigma $-algebra generated by all predictable
subsets;

\begin{itemize}
\item $L^{2}(B,\mathbf{B,}\lambda ;%
\mathbb{R}
^{n\times l})$ denotes the set of all $\mathbf{B}$-measurable functions $%
\Phi (\cdot )$ valued in $%
\mathbb{R}
^{n\times l}$ such that%
\begin{equation*}
\left\Vert \Phi \right\Vert ^{2}=\dint\nolimits_{B}\left\vert \Phi
(e)\right\vert ^{2}\lambda (de)<\infty .
\end{equation*}
\end{itemize}

The (adapted) solution of a random horizon BSDE with jumps is defined as a
triple of processes $(Y_{t},Z_{t},U_{t}),$ $t\in \lbrack 0,\tau ],$
belonging to $S^{2}\times L^{2}(W)\times L^{2}(\widetilde{\mu })$ such that 
\begin{equation}
Y_{t}=\xi _{\tau }+\dint\nolimits_{t\wedge \tau }^{\tau }f(\omega
,s,Y_{s},Z_{s},U_{s})ds-\dint\nolimits_{t\wedge \tau }^{\tau
}Z_{s}dW_{s}-\dint\nolimits_{t\wedge \tau }^{\tau }\dint\nolimits_{B}U_{s}(e)%
\widetilde{\mu }(ds,de),  \label{BSDE jumps}
\end{equation}

where the terminal condition $\xi \in L^{2}(\Omega ,\tciFourier _{\tau },P)$
and the generator $f:\Omega \times \lbrack 0,\infty )\times 
\mathbb{R}
^{n}\times 
\mathbb{R}
^{n\times d}\times L^{2}(B,\mathbf{B,}\lambda ;%
\mathbb{R}
^{n\times l})\rightarrow 
\mathbb{R}
^{n}$ is $\mathbf{P}\otimes \mathbf{B}(%
\mathbb{R}
^{n})\otimes \mathbf{B}(%
\mathbb{R}
^{n\times d})\otimes \mathbf{B}(L^{2}(B,\mathbf{B},\lambda ;%
\mathbb{R}
^{n\times l}\mathbf{)\mathbf{)}}$-measurable.

\bigskip

We will denote by $Y_{t}(f,\xi _{\tau })$ the first component of the
solution of BSDE (\ref{BSDE jumps}). In the sequel, we consider the
following assumptions on generator $f:$

\begin{itemize}
\item[(A1)] 
\begin{equation*}
E\left[ \dint\nolimits_{0}^{\tau }\left\vert f(\omega ,t,0,0,0)\right\vert
^{2}dt\right] <\infty ;
\end{equation*}

\item[(A2)] $P-a.s$., for any $Y,Y^{\prime }\in 
\mathbb{R}
^{n}$, $Z,Z^{\prime }\in 
\mathbb{R}
^{n\times d}$, $U,U^{\prime }\in L^{2}(B,\mathbf{B},\lambda ;%
\mathbb{R}
^{n\times l}\mathbf{)}$, $t\geq 0,$

\begin{eqnarray*}
\left\vert f(\omega ,t,Y,Z,U)-f(\omega ,t,Y^{\prime },Z^{\prime },U^{\prime
})\right\vert &\leq &u_{1}(t)\left\vert Y-Y^{\prime }\right\vert \\
&&+u_{2}(t)\left( \left\vert Z-Z^{\prime }\right\vert +\left\Vert
U-U^{\prime }\right\Vert \right) ,
\end{eqnarray*}%
where $u_{1}(t)$ and $u_{2}(t)$ are nonnegative, deterministic functions
satisfying 
\begin{equation}
\dint\nolimits_{0}^{\infty }u_{1}(t)dt+\dint\nolimits_{0}^{\infty
}u_{2}(t)^{2}dt<\infty .  \label{Lipschtiz}
\end{equation}

\item[(A3)] $P-a.s$, $\forall (Y,Z,U)\in 
\mathbb{R}
^{n}\times 
\mathbb{R}
^{n\times d}\times L^{2}(B,\mathbf{B,}\lambda ;%
\mathbb{R}
^{n\times l}),$ $t\rightarrow f(\omega ,t,Y,Z,U)$ is continuous in $t\in
\lbrack 0,\tau ].$

\item[(A4)] $P-a.s$, $\forall Y\in 
\mathbb{R}
^{n},\forall Z\in 
\mathbb{R}
^{n\times d},\forall U,U^{\prime }\in L^{2}\mathbf{(}B,\mathbf{B},\lambda ;%
\mathbb{R}
^{n\times l}\mathbf{)},$ we have%
\begin{equation*}
\left\vert f(\omega ,t,Y,Z,U)-f(\omega ,t,Y,Z,U^{\prime })\right\vert \leq
\left\vert \dint\nolimits_{B}\left( U(e)-U^{\prime }(e)\right) \overline{%
\gamma }_{t}(\omega ,e)^{\top }\lambda (de)\right\vert ,
\end{equation*}%
where $\overline{\gamma }:\Omega \times \lbrack 0,\infty )\times
B\rightarrow 
\mathbb{R}
^{n\times l}$ is $\mathbf{P}\otimes \mathbf{B}$-measurable and satisfies
\end{itemize}

$\tint\nolimits_{B}\left\vert \overline{\gamma }_{t}(\omega ,e)\right\vert
^{2}\lambda (de)\leq u_{2}(t)^{2},$ and where $\overline{\gamma }%
_{t}^{i}(\omega ,e),$ $i=1,...,l,$ is the $i$th component of $\overline{%
\gamma }_{t}(\omega ,e)$ satisfying $\overline{\gamma }_{t}^{i}(\omega
,e)\geq -1$, $P-a.s.$

\subsection{Preliminary results}

This section presents existence, uniqueness and comparison results for
adapted solutions to BSDEs with jumps and with random terminal times.

\begin{theorem}[Yin and Situ \protect\cite{cc Yin and Sityu}]
\label{Theorem BBP} Let $\xi _{\tau }$ and $f$ be the terminal condition and
the generator of BSDE (\ref{BSDE jumps}), respectively, and let $f$ satisfy
(A1) and (A2). Then there exists a unique solution to BSDE (\ref{BSDE jumps}%
).
\end{theorem}

In preparation of the main theorem, we recall the comparison theorem for
one-dimensional BSDEs with jumps. Therefore, from now on, we assume that $%
n=1.$\ In order to get a comparison result, one must impose a control on the
size of the jumps\footnote{%
See Royer \cite{ccRoyer} for a counter-example when only Lipschitz
continuity is assumed.}.\ This is the reason why (A4) must hold in addition
to (A1) and (A2).

\begin{theorem}[Yin and Mao \protect\cite{ccYin and Mao}]
\label{comparison theorem royer} Consider two generators $f_{1}$ and $f_{2}$
verifying (A1), (A2) and (A4). Let $\xi _{\tau }^{1},\xi _{\tau }^{2}\in
L^{2}(\Omega ,\tciFourier _{\tau },P)$\ be two terminal conditions for BSDEs
driven respectively by $f_{1}$ and $f_{2}.$ Denote by $%
(Y_{t}^{1},Z_{t}^{1},U_{t}^{1})$ and $(Y_{t}^{2},Z_{t}^{2},U_{t}^{2}),$ $%
t\in \lbrack 0,\tau ]$, the respective solutions of these equations.\ If $%
\xi _{\tau }^{1}\leq \xi _{\tau }^{2}$\ $P-a.s.$ and $%
f_{1}(t,Y_{t}^{1},Z_{t}^{1},U_{t}^{1})\leq
f_{2}(t,Y_{t}^{1},Z_{t}^{1},U_{t}^{1})$ $P-a.s.,$ then almost surely $%
Y_{t}^{1}\leq Y_{t}^{2}$, for all $t\in \lbrack 0,\tau ].$
\end{theorem}

Before stating the main theorem, we also need the following lemma. It is a
strict comparison theorem for BSDEs with jumps and with the same terminal
condition.

\begin{lemma}
\label{last lemma}Let the generators $f_{1}$ and $f_{2}$ verify (A1), (A2)
and (A4), and let the terminal condition $\xi _{\tau }\in L^{2}(\Omega
,\tciFourier _{\tau },P).$\ Denote by $(Y_{t}^{1},Z_{t}^{1},U_{t}^{1})$ and $%
(Y_{t}^{2},Z_{t}^{2},U_{t}^{2}),$ $t\in \lbrack 0,\tau ]$, the respective
solutions of the BSDEs with generators $f_{1}$ and $f_{2}$ and terminal
conditions $\xi _{\tau }$.\ If $%
f_{2}(t,Y_{t}^{1},Z_{t}^{1},U_{t}^{1})<f_{1}(t,Y_{t}^{1},Z_{t}^{1},U_{t}^{1}) 
$ $P-a.s.,$ then almost surely$,$ $Y_{t}^{2}<Y_{t}^{1}$, for all $t\in
\lbrack 0,\tau ).$
\end{lemma}

\begin{proof}
The proof follows from the same argument as in the proof of the comparison
theorem for classical BSDEs of El Karoui et al. \cite{ccEl Karoui et al}
(for the jump diffusion case, see the proofs of theorem 3.1 and corollary
3.1 of Yin and Mao \cite{ccYin and Mao}).
\end{proof}

\section{A converse comparison theorem}

We now introduce the infinite horizon decoupled forward backward stochastic
differential equations (FBSDEs, for short) for which we will prove a
converse comparison result.\ 

For any given $(t,x)\in \lbrack 0,\infty )\times 
\mathbb{R}
^{m}$, consider the following FBSDEs%
\begin{eqnarray}
X_{s}^{t,x}
&=&x+\dint\nolimits_{t}^{s}a(r,X_{r}^{t,x})dr+\dint%
\nolimits_{t}^{s}b(r,X_{r}^{t,x})dW_{r}  \notag \\
&&+\dint\nolimits_{t}^{s}\dint\nolimits_{B}c(r,X_{r^{-}}^{t,x},e)\widetilde{%
\mu }(dr,de),  \label{FBSDEI} \\
Y_{s} &=&h(X_{\infty }^{t,x})+\dint\nolimits_{s}^{\infty
}f_{1}(r,X_{r}^{t,x},Y_{r},Z_{r},\dint\nolimits_{B}U_{r}(e)\gamma
_{r}(e)^{\top }\lambda (de))dr  \notag \\
&&-\dint\nolimits_{s}^{\infty }Z_{r}dW_{r}-\dint\nolimits_{s}^{\infty
}\dint\nolimits_{B}U_{r}(e)\widetilde{\mu }(dr,de),  \label{FBSDEII}
\end{eqnarray}%
where $s\in \lbrack t,\infty )$, $\gamma $ is assumed to be deterministic
and $\gamma :[0,\infty )\times B\rightarrow 
\mathbb{R}
^{1\times l}$ is $\mathbf{B}(%
\mathbb{R}
^{+})\otimes \mathbf{B}$-measurable and satisfies

$\tint\nolimits_{B}\left\vert \gamma _{t}(e)\right\vert ^{2}\lambda (de)\leq
u_{2}(t)^{2},$ where $u_{2}(t)$\ is as in (\ref{Lipschtiz}) and $\gamma
_{t}^{i}(e),$ $i=1,...,l,$ is the $i$th component of $\gamma _{t}(e)$
satisfying $\gamma _{t}^{i}(e)\geq -1.$

From now on, we assume that $a,b,c,f_{1}$ and $h$ in (\ref{FBSDEI}) and (\ref%
{FBSDEII}) are deterministic. We assume that $a:[0,\infty )\times 
\mathbb{R}
^{m}\rightarrow 
\mathbb{R}
^{m}$, $b:[0,\infty )\times 
\mathbb{R}
^{m}\rightarrow 
\mathbb{R}
^{m\times d}$ and $c:[0,\infty )\times 
\mathbb{R}
^{m}\times B\rightarrow 
\mathbb{R}
^{m\times l}$\ are continuous and such that they guarantee the existence and
uniqueness of a strong solution to (\ref{FBSDEI}). We also assume that $%
X_{\infty }^{t,x}$ exists and that $E\left[ \left\vert X_{\infty
}^{t,x}\right\vert ^{2}\right] <\infty $.

Regarding (\ref{FBSDEII}), we will suppose that the real valued functions $h:%
\mathbb{R}
^{m}\rightarrow 
\mathbb{R}
$ and $f_{1}:[0,\infty )\times 
\mathbb{R}
^{m}\times 
\mathbb{R}
\times 
\mathbb{R}
^{1\times d}\times 
\mathbb{R}
\rightarrow 
\mathbb{R}
$ are $\mathbf{B}(%
\mathbb{R}
^{m})$-measurable and $\mathbf{B}(%
\mathbb{R}
^{+})\otimes \mathbf{B}(%
\mathbb{R}
^{m})\otimes \mathbf{B}(%
\mathbb{R}
)\otimes \mathbf{B}(%
\mathbb{R}
^{1\times d})\otimes \mathbf{B}(%
\mathbb{R}
\mathbf{\mathbf{)}}$-measurable, respectively, and that they are Lipschitz
continuous with respect to $X$, with Lipschitzian function satisfying (\ref%
{Lipschtiz}). Moreover, we assume that for every$\ X\in 
\mathbb{R}
^{m}$, $f(t,X,Y,Z,U)=f_{1}(t,X,Y,Z,\tint\nolimits_{B}U(e)\gamma
_{t}(e)^{\top }\lambda (de))$ satisfies assumptions (A1), (A2), (A3) and
(A4), and we assume that there exists a constant $C$ such that for every$\
X\in 
\mathbb{R}
^{m}$, $\left\vert h(X)\right\vert \leq C(1+\left\vert X\right\vert )$, so
that $h(X_{\infty }^{t,x})\in L^{2}(\Omega ,\tciFourier ,P)$ and (\ref%
{FBSDEII}) admits a unique solution. We denote the solution by $%
(Y_{r}^{t,x,1},Z_{r}^{t,x,1},U_{r}^{t,x,1}),$ where $Y_{r}^{t,x,1}$ is
adapted and $(Z_{r}^{t,x,1},U_{r}^{t,x,1})$ are predictable with respect to
the filtration $(\tciFourier _{r}^{t},r\geq t),$ where $\tciFourier
_{r}^{t}=\sigma \lbrack W_{s}-W_{t},\mu ((t,s],A);t\leq s\leq r,A\in \mathbf{%
B}]\vee N.$ Then,%
\begin{equation*}
(u^{1}(t,x),\chi ^{1}(t,x),\zeta
^{1}(t,x)):=(Y_{t}^{t,x,1},Z_{t}^{t,x,1},\dint\nolimits_{B}U_{t}^{t,x,1}(e)%
\gamma _{t}(e)^{\top }\lambda (de))
\end{equation*}%
is deterministic, and from the uniqueness of solution to (\ref{FBSDEII}), it
is known (see for example Barles et al.\ \cite{cc barles buckdhan and
pardoux}, p. 69) that for any $s\in \lbrack t,\infty ),$%
\begin{equation}
Y_{s}^{t,x,1}=Y_{s}^{s,X_{s}^{t,x},1}=u^{1}(s,X_{s}^{t,x}).  \label{ccY}
\end{equation}%
Since the same reasoning holds for $\chi $\ and $\zeta $, we similarly have%
\begin{eqnarray}
Z_{s}^{t,x,1} &=&\chi ^{1}(s,X_{s}^{t,x}),  \label{ccZ} \\
\dint\nolimits_{B}U_{s}^{t,x,1}(e)\gamma _{s}(e)^{\top }\lambda (de)
&=&\zeta ^{1}(s,X_{s-}^{t,x}).  \label{ccU}
\end{eqnarray}

We can now state the converse comparison theorem for the infinite horizon
decoupled FBSDEs (\ref{FBSDEI})-(\ref{FBSDEII}):

\begin{theorem}[Converse comparison theorem]
Suppose that $f_{1}$ is continuous and let $a,b,c,f_{1}$ and $h$ satisfy the
above assumptions, so that there exists a unique solution to (\ref{FBSDEI})
and (\ref{FBSDEII}). Denote by $Y_{s}^{t,x,1}(\xi _{v})$ the first component
of the solution of (\ref{FBSDEII}) with generator $f_{1}$ and terminal
condition $\xi _{v}$ at time $v$, where $v$ is a stopping time and $t\leq
s<v.$ Assume further that $Y_{t}^{t,x,1}(h(X_{\infty }^{t,x}))=u^{1}(t,x)\in
C^{1,2}([0,\infty )\times 
\mathbb{R}
^{m}).$ If, for any given $(t,x)\in \lbrack 0,\infty )\times 
\mathbb{R}
^{m}$, $\exists \delta >0$ such that 
\begin{equation}
Y_{t}^{t,x,1}(u^{1}(v,X_{v}^{t,x}))\leq Y_{t}^{t,x,2}(u^{1}(v,X_{v}^{t,x})),
\label{inv com}
\end{equation}%
for every stopping time $v$ such that $t\leq v\leq t+\delta $, and where $%
Y_{t}^{t,x,2}(u^{1}(v,X_{v}^{t,x}))$ is the first component of the time $t$
solution of a FBSDE with driver $f_{2}$ satisfying the same assumptions as $%
f_{1}$\ and terminal condition $u^{1}(v,X_{v}^{t,x})$ at time $v$, then 
\begin{equation*}
f_{1}(t,x,u^{1}(t,x),\chi ^{1}(t,x),\zeta ^{1}(t,x))\leq
f_{2}(t,x,u^{1}(t,x),\chi ^{1}(t,x),\zeta ^{1}(t,x)).
\end{equation*}
\end{theorem}

\begin{proof}
By contradiction, suppose that 
\begin{equation}
f_{1}(t,x,u^{1}(t,x),\chi ^{1}(t,x),\zeta
^{1}(t,x))>f_{2}(t,x,u^{1}(t,x),\chi ^{1}(t,x),\zeta ^{1}(t,x)).
\label{f ineq first}
\end{equation}%
By assumption, $f,a,b$ and $c$ are continuous, and $u^{1}$ is assumed to be
continuous and smooth, so it follows from the Feynman-Kac formula (see also
theorem 3.4 in Barles et al. \cite{cc barles buckdhan and pardoux}) that $%
\chi ^{1}$ and $\zeta ^{1}$ are continuous functions, since $\chi
^{1}(t,x)=u_{x}^{1}(t,x)^{\top }b(t,x)$ and $\zeta
^{1}(t,x)=\tint\nolimits_{B}(u^{1}(t,x+c(t,x,e))-u^{1}(t,x))\gamma
_{t}(e)^{\top }\lambda (de)$ (the letter $x$ as lower indice indicates
differentiation, as usual). As a consequence, there exists $0<\eta <\infty $
such that for any $(s,k)\in \lbrack t,\infty )\times 
\mathbb{R}
^{m}$ satisfying $t\leq s\leq t+\eta ,$ $\left\vert x-k\right\vert \leq \eta
,$ 
\begin{equation}
f_{1}(s,k,u^{1}(s,k),\chi ^{1}(s,k),\zeta
^{1}(s,k))>f_{2}(s,k,u^{1}(s,k),\chi ^{1}(s,k),\zeta ^{1}(s,k)).
\label{cc prolonged ineq}
\end{equation}%
Define now the stopping time%
\begin{equation*}
\tau :=\inf \{s>t:\left\vert X_{s}^{t,x}-x\right\vert >\eta \}\wedge (t+\eta
)\wedge (t+\delta ),
\end{equation*}%
and let%
\begin{equation*}
(\overline{Y}_{s}^{1},\overline{Z}_{s}^{1},\overline{U}%
_{s}^{1}):=(Y_{s}^{t,x,1},Z_{s}^{t,x,1},U_{s}^{t,x,1}),\text{ }t\leq s\leq
\tau .
\end{equation*}%
Then$,$ $(\overline{Y}_{s}^{1},\overline{Z}_{s}^{1},\overline{U}_{s}^{1})$
is solution of the following FBSDE:%
\begin{eqnarray*}
\overline{Y}_{s}^{1} &=&u^{1}(\tau ,X_{\tau
}^{t,x})+\dint\nolimits_{s}^{\tau }f_{1}(r,X_{r}^{t,x},\overline{Y}_{r}^{1},%
\overline{Z}_{r}^{1},\dint\nolimits_{B}\overline{U}_{r}^{1}(e)\gamma
_{r}(e)^{\top }\lambda (de))dr \\
&&-\dint\nolimits_{s}^{\tau }\overline{Z}_{r}^{1}dW_{r}-\dint\nolimits_{s}^{%
\tau }\dint\nolimits_{B}\overline{U}_{r}^{1}(e)\widetilde{\mu }(dr,de).
\end{eqnarray*}%
Now, consider the FBSDE%
\begin{eqnarray}
\overline{Y}_{s}^{2} &=&u^{1}(\tau ,X_{\tau
}^{t,x})+\dint\nolimits_{s}^{\tau }f_{2}(r,X_{r}^{t,x},\overline{Y}_{r}^{2},%
\overline{Z}_{r}^{2},\dint\nolimits_{B}\overline{U}_{r}^{2}(e)\gamma
_{r}(e)^{\top }\lambda (de))dr  \label{BSDE2} \\
&&-\dint\nolimits_{s}^{\tau }\overline{Z}_{r}^{2}dW_{r}-\dint\nolimits_{s}^{%
\tau }\dint\nolimits_{B}\overline{U}_{r}^{2}(e)\widetilde{\mu }(dr,de), 
\notag
\end{eqnarray}%
where 
\begin{equation*}
(\overline{Y}_{s}^{2},\overline{Z}_{s}^{2},\overline{U}_{s}^{2}):=(\overline{%
Y}_{s}^{t,x,2},\overline{Z}_{s}^{t,x,2},\overline{U}_{s}^{t,x,2}),\text{ }%
t\leq s\leq \tau .
\end{equation*}
$u^{1}(\tau ,X_{\tau }^{t,x})$\ is the first component of the time $\tau $
solution of a BSDE with coefficient $f_{1}$ and terminal condition $%
h(X_{\infty }^{t,x});$ therefore, $u^{1}(\tau ,X_{\tau }^{t,x})$ is square
integrable, that is, $E\left[ u^{1}(\tau ,X_{\tau }^{t,x})^{2}\right]
<\infty .$ By assumption on $f_{2}$, it follows that there exists a unique
solution to (\ref{BSDE2}).

Taking (\ref{ccY}), (\ref{ccZ}) and (\ref{ccU}) into account, together with (%
\ref{cc prolonged ineq}), we can apply the comparison theorem of lemma \ref%
{last lemma} and get%
\begin{equation}
\overline{Y}_{t}^{1}(u^{1}(\tau ,X_{\tau }^{t,x}))>\overline{Y}%
_{t}^{2}(u^{1}(\tau ,X_{\tau }^{t,x})).  \label{contradiction}
\end{equation}%
which, by uniqueness, yields that 
\begin{equation*}
Y_{t}^{t,x,1}(u^{1}(\tau ,X_{\tau }^{t,x}))>Y_{t}^{t,x,2}(u^{1}(\tau
,X_{\tau }^{t,x})),
\end{equation*}%
which contradicts (\ref{inv com}). Therefore, we must have that 
\begin{equation*}
f_{1}(t,x,u^{1}(t,x),\chi ^{1}(t,x),\zeta ^{1}(t,x))\leq
f_{2}(t,x,u^{1}(t,x),\chi ^{1}(t,x),\zeta ^{1}(t,x)).
\end{equation*}
\end{proof}


\begin{thebibliography}{99}
\bibitem{cc barles buckdhan and pardoux} Barles, G., Buckdhan, R. and E.
Pardoux, 1997.\ Backward stochastic differential equations and
integral-partial differential equations.\ \textit{Stochastics and
Stochastics Reports} 60, 57$-$83.

\bibitem{Coquet et al} Briand, P., Coquet, F., Hu, Y., M\'{e}min, J. and S.\
Peng, 2000. A converse comparison theorem for BSDEs and related properties
of $g$-expectation. \textit{Electronic Communications in Probability} 5,
101--117.

\bibitem{Chen} Chen, Z., 1998.\ A property of backward stochastic
differential equations.\ \textit{Comptes Rendus de l'Acad\'{e}mie des
Sciences de Paris}, S\'{e}rie I 326, 483$-$488.

\bibitem{Coquet Hu et al} Coquet, F., Hu, Y., M\'{e}min, J. and S. Peng,
2001.\ A general converse comparison theorem for backward stochastic
differential equations.\ \textit{Comptes Rendus de l'Acad\'{e}mie des
Sciences de Paris}, S\'{e}rie I 333, 577$-$581.

\bibitem{ccEl Karoui et al} El Karoui, N., Peng, S. and M. C. Quenez, 1997.
Backward stochastic differential equations in finance. \textit{Mathematical
Finance} 7, 1$-$71.

\bibitem{Jiang I} Jiang, L., 2004.\ Some results on the uniqueness of
generators of backward stochastic differential equations.\ \textit{Comptes
Rendus de l'Acad\'{e}mie des Sciences de Paris}, S\'{e}rie I 338, 575$-$580.

\bibitem{cc Jiang} Jiang, L., 2005. Converse comparison theorems for
backward stochastic differential equations. \textit{Statistics and
Probabilty Letters} 71, 173$-$183.

\bibitem{Pardoux and Peng} Pardoux, E., Peng, S.,\ 1990. Adapted solution of
a backward stochastic differential equation. \textit{Systems and Control
Letters} 14, 55$-$61.

\bibitem{cc Peng} Peng, S., 1992.\ Stochastic Hamilton-Jacobi-Bellman
equations.\ \textit{SIAM Journal on\ Control and Optimization} 30, 284$-$304.

\bibitem{ccRoyer} Royer, M.,\ 2006. Backward stochastic differential
equations with jumps and related non-linear expectations.\ \textit{%
Stochastic Processes and their Applications} 116, 1358$-$1376.

\bibitem{ccTang and Li} Tang, S., Li, X., 1994. Necessary conditions for
optimal control of stochastic systems with random jumps.\ \textit{SIAM
Journal on Control and Optimization}\ 32, 1447$-$1475.

\bibitem{ccYin and Mao} Yin, J. and X. Mao, 2008.\ The adapted solution and
comparison theorem for backward stochastic differential equations with
Poisson jumps and applications. \textit{Journal of Mathematical Analysis and
Applications}\ 346, 345$-$358.

\bibitem{cc Yin and Sityu} Yin, J.\ and R. Situ, 2003. On solutions of
forward-backward stochastic differential equation with Poisson jumps.
Stochastic Analysis and Applications 23, 1419$-$1448.
\end{thebibliography}
\end{document}